\documentclass[12pt]{article}
\usepackage[english,activeacute]{babel}
\usepackage{amsmath,amsfonts,amssymb,amstext,amsthm,amscd,mathrsfs,amsbsy}
\usepackage{xypic}

\newtheorem{teo}{Theorem}[section]
\newtheorem{pro}[teo]{Proposition}
\newtheorem{coro}[teo]{Corollary}
\newtheorem{lem}[teo]{Lemma}

\theoremstyle{definition}
\newtheorem{defi}[teo]{Definition}
\newtheorem{exam}[teo]{Example}
\newtheorem{rem}[teo]{Remark}

\newcommand{\V}{\mathbf V}

\newcommand{\N}{\mathbb N}

\newcommand{\C}{\mathbb C}
\newcommand{\m}{\mathfrak{m}}

\newcommand{\bb}{\mathfrak{b}}
\newcommand{\ap}{\mathfrak{a}_p}
\newcommand{\ida}{\mathfrak{a}}
\newcommand{\Jn}{\mathcal{J}_n}

\newcommand{\Ox}{\mathcal O}

\newcommand{\Li}{\mathcal L_n(X,0)}

\newcommand{\Po}{\C[x_1,\ldots,x_s]}

\newcommand{\Na}{Nash_n(X)}

\newcommand{\rk}{\mbox{rank}}

\newcommand{\Sp}{\mbox{Spec }}
\newcommand{\Img}{\mbox{Im}}
\newcommand{\Jc}{\mbox{Jac}}
\newcommand{\Jcn}{\mbox{Jac}_n(F)}
\newcommand{\Jcd}{\mbox{Jac}_2(F)}

\begin{document}

\title{Computational aspects of the higher Nash blowup of hypersurfaces}

\author{Daniel Duarte\footnote{Research supported by CONACyT (M\'{e}xico) through the program
\textit{C\'atedras para J\'ovenes Investigadores}.}}

\maketitle

\begin{abstract}
The higher Nash blowup of an algebraic variety replaces singular points with limits of certain spaces
carrying higher-order data associated to the variety at non-singular points. In this note we will
define a higher-order Jacobian matrix that will allow us to make explicit computations concerning
the higher Nash blowup of hypersurfaces. Firstly, we will generalize a known method to compute the
fiber of this modification. Secondly, we will give an explicit description of the ideal
whose blowup gives the higher Nash blowup. As a consequence, we will deduce a higher-order
version of Nobile's theorem for normal hypersurfaces.
\end{abstract}

\section*{Introduction}

The main purpose of this note is to present, as the title suggests, some computational aspects
of the higher Nash blowup of a hypersurface. The higher Nash blowup is defined as follows
(see \cite{No}, \cite{OZ}, \cite{Y}):
\\
\\
Let $X=\V(I)\subset\C^s$ be an irreducible algebraic variety of dimension
$d$, given as the zero set of some ideal $I$. Let $R$ be its ring of regular functions. For each $p\in X$,
let $(R_p,\m_p)$ be the localization at $p$ and define the $R_p/\m_p\cong\C-$vector space
$T^n_pX:=(\m_p/\m_p^{n+1})^{\vee}.$ This is a vector space of dimension $D=\binom{d+n}{d}-1$,
whenever $p$ is a non-singular point. The fact that $X\subset\C^s$ implies that
$T^n_pX\subset T^n_p\C^s\cong\C^E$, where $E=\binom{s+n}{s}-1$, that is, we can see $T^n_pX$ as an element
of the Grassmanian $Gr(D,\C^E)$. Let $S(X)$ be the singular locus of $X$. Now consider the Gauss map:
$$G_n:X\setminus S(X)\rightarrow Gr(D,\C^E),\mbox{ }\mbox{ }\mbox{ }p\mapsto T^n_pX.$$
Denote by $\Na$ the Zariski closure of the graph of $G_n$. Call $\pi_n$ the restriction to $\Na$ of
the projection of $X\times Gr(D,\C^E)$ to $X$. When $n=1$, the pair $(\Na,\pi_n)$ is usually called the
\textit{Nash modification} of $X$. For $n>1$, $(\Na,\pi_n)$ is called the \textit{higher Nash blowup} of $X$.
This construction gives a canonical modification of an algebraic variety that replaces singular points
by limits of sequences $\{T^n_{p_i}X\}$, where $\{p_i\}\subset X$ is any sequence of non-singular points
converging to a singular one.
\\
\\
Unfortunately, despite of being a natural and geometrically attractive modification, it is hard to compute
in general. The goal of this note is to deal with this problem, to some extent, in the case of hypersurfaces.
We will start by defining in Section 1 a generalization of the Jacobian matrix that involves
also higher-order derivatives, which is more suitable to this context. Using this matrix, we will give in
Section 2 some higher-order criteria of non-singularity. Next, we will prove in Section 3
that the spaces $T^n_pX$ can be identified with the kernel of the higher-order Jacobian, as with the
tangent space.
\\
\\
In the last section we will give some applications of the previous results. Firstly, we will generalize
a method proposed by D. O'Shea which computes limits of tangent spaces to a singular point of a hypersurface
(see \cite{Sh}). This method, along with the theory of Gr\"obner bases, will allow us to compute
examples showing some interesting phenomena of the set of limits of spaces $T^n_pX$. Later,
using some results of O. Villamayor appearing in \cite{V}, we will explicitly describe the ideal whose
blowup gives the higher Nash blowup by means of the higher-order Jacobian matrix.
\\
\\
As a final application, we will study a higher-order version of the following theorem due to A. Nobile:
the Nash modification of a variety is an isomorphism if and only if the variety is non-singular (see \cite{No}).
We will prove the analogous statement for the higher-order Nash blowup of normal hypersurfaces.
To that end, we will show that the singular locus of a hypersurface coincides with the zero set of
the ideal whose blowup gives the higher Nash blowup. We will also compute some examples of
singular plane curves where the second-order analogue of Nobile's theorem holds as well.


\section{A higher-order Jacobian matrix}\label{s. higher-order jacobian}

The first thing we are going to do is to define a higher-order version of the Jacobian
matrix of a polynomial. We begin by presenting an example to illustrate the idea of
the definition.
\\
\\
Let $F(x,y)=x^3-y^2\in\C[x,y]$. Let $p=(a,b)\in X=\V(F)$. The Jacobian matrix of $F$
evaluated at $p$ is defined as:
\[
\Jc(F)_{|p}:=
\begin{pmatrix}
3x^2& 	-2y
\end{pmatrix}_{|p}
\]
We want to define another matrix involving also higher-order derivatives that generalizes the Jacobian matrix.
Let $\ida_p=\langle x-a,y-b \rangle\subset\C[x,y]$. Consider the following linear map:
\begin{align}
\theta:\ap&\rightarrow\C^5\notag\\
f&\mapsto \Big(\frac{\partial f}{\partial x},\frac{\partial f}{\partial y},
\frac{1}{2!}\frac{\partial^2 f}{\partial x^2},\frac{\partial^2 f}{\partial x \partial y},
\frac{1}{2!}\frac{\partial^2 f}{\partial y^2}\Big)_{|p}.\notag
\end{align}
Let $\bb=\langle F \rangle$. Notice that $\bb\subset\ap$. Let $g\cdot F\in\bb$, where
$g\in\C[x,y]$. Using repeatedly the Leibniz rule and the fact $F(p)=0$, we can write
$\theta(gF)$ as follows:
\begin{align}
\theta(gF)=&g(p)\cdot(3x^2,-2y,3x,0,-1)_{|p}+\notag\\
&\frac{\partial g}{\partial x}(p)\cdot(F,0,3x^2,-2y,0)_{|p}+\notag\\
&\frac{\partial g}{\partial y}(p)\cdot(0,F,0,3x^2,-2y)_{|p}.\notag
\end{align}
Let
\begin{align}\label{e. Jac2(x3-y2)}
\Jc_2(F):=
\begin{pmatrix}
3x^2& 	-2y& 	3x&  	 0&     -1\\
F&  	 0&    3x^2&   -2y&      0\\
0&       F&     0&      3x^2&   -2y
\end{pmatrix}
\end{align}
Thus
\[
\theta(gF)=\Jc_2(F)^t|_p\cdot
\begin{pmatrix}
g(p)\\
\frac{\partial g}{\partial x}(p)\\
\frac{\partial g}{\partial y}(p)
\end{pmatrix}
\]
We call $\Jc_2(F)$ the \textit{Jacobian matrix of order 2} of $F$.
Now we proceed exactly as in this example to define a higher-order Jacobian of a polynomial.
First recall the multi-index notation. Let $\alpha=(\alpha_1,\ldots,\alpha_s)$, $\beta=(\beta_1,\ldots,\beta_s)\in\N^s$:
\begin{itemize}
\item $\alpha\leq\beta\Leftrightarrow\alpha_i\leq\beta_i\mbox{ }\forall i\in\{1,\ldots,s\}$.
\item $|\alpha|=\alpha_1+\cdots+\alpha_s$.
\item $\alpha!=\alpha_1!\cdot\alpha_2!\cdots\alpha_s!$.
\item $\binom{\alpha}{\beta}=\binom{\alpha_1}{\beta_1}\binom{\alpha_2}{\beta_2}\cdots
\binom{\alpha_s}{\beta_s}=\frac{\alpha!}{\beta!(\alpha-\beta)!}$.
\item $\partial^{\alpha}=\partial^{\alpha_1}\partial^{\alpha_2}\cdots\partial^{\alpha_s}$.
\end{itemize}
Using this notation, the general Leibniz rule states that
$$\partial^{\alpha}(g\cdot f)=
\sum_{\{\beta|\beta\leq\alpha\}}\binom{\alpha}{\beta}\partial^{\alpha-\beta}f\partial^{\beta}g$$
for any $f,g\in\Po$. If we define $\partial^{\alpha-\beta}f=0$ when $\alpha_i<\beta_i$ for some
$1\leq i \leq s$, then the general Leibniz rule can also be written as:
\begin{align}\label{e. general Leibniz rule}
\partial^{\alpha}(g\cdot f)=
\sum_{\{\beta|0\leq|\beta|\leq|\alpha|\}}\binom{\alpha}{\beta}\partial^{\alpha-\beta}f\partial^{\beta}g.
\end{align}
Let $F\in\Po$ and $p=(a_1,\ldots,a_s)\in X=\V(F)\subset\C^s$.
Let $\ida_p=\langle x_1-a_1,\ldots,x_s-a_s\rangle\subset\Po$. Fix $n\in\N$.
Let $N=\binom{n+s}{s}$ and consider the following linear map:
\begin{align}\label{e. theta}
\theta:\ap&\rightarrow\C^{N-1}\\
f&\mapsto\Big(\frac{\partial^{\alpha}f}{\alpha!}|1\leq|\alpha|\leq n\Big)_{|p}.\notag
\end{align}
We arrange this vector increasingly using graded lexicographical order, where $\alpha_1<\alpha_2<\ldots<\alpha_s$.
\\
\\
Let $\bb=\langle F \rangle$. Notice that $\bb\subset\ap$. Let $g\cdot F\in\bb$, where $g\in\Po$.
Using the general Leibniz rule (\ref{e. general Leibniz rule}) and the fact $F(p)=0$,
we can write $\theta(gF)$ as follows (recall that we defined $\partial^{\alpha-\beta}F=0$ if $\alpha_i<\beta_i$
for some $i$):
\begin{equation}\label{e. theta(gF)}
\theta(gF)=\sum_{\{\beta|0\leq|\beta|\leq n-1\}}\partial^{\beta}g(p)\cdot
\Big(\binom{\alpha}{\beta}\frac{\partial^{\alpha-\beta}F}{\alpha!}|1\leq|\alpha|\leq n\Big)_{|p}.
\end{equation}
Let $r_{\beta}:=\beta!\cdot\Big(\binom{\alpha}{\beta}
\frac{\partial^{\alpha-\beta}F}{\alpha!}|1\leq|\alpha|\leq n\Big)$, where $\beta$ is such that $0\leq|\beta|\leq n-1$.
We multiply by $\beta!$ to obtain later some nice properties among these vectors
(see lemma \ref{l. remarks c-beta} below). As before, we arrange $r_{\beta}$
using graded lexicographical order on $\alpha$. There are $M=\binom{n+s-1}{s}$ such vectors.

\begin{lem}\label{l. remarks c-beta}
Fix $\beta$ such that $0\leq|\beta|\leq n-1$.
\begin{itemize}
\item[(i)]The $\alpha-$entry of $r_{\beta}$ satisfies:
\[r_{\beta,\alpha}=
\left\{
\begin{array}{rll}
&0,             & \mbox{ if }\alpha_i<\beta_i \mbox{ for some } 1\leq i \leq s,\\
&\frac{\partial^{\alpha-\beta}F}{(\alpha-\beta)!},  & \mbox{ if }\alpha\geq\beta.
\end{array}
\right.
\]
In particular, $r_{(0,\ldots,0)}=\Big(\frac{\partial^{\alpha}F}{\alpha!}|1\leq|\alpha|\leq n\Big)$.
\item[(ii)] Let $\alpha$ be such that $1\leq|\alpha|\leq n-|\beta|$. Then,
$r_{\beta,\beta+\alpha}=r_{(0,\ldots,0),\alpha}$.
\end{itemize}
Now assume $1\leq|\beta|\leq n-1$.
\begin{itemize}
\item[(iii)]If $\alpha\neq\beta+\alpha'$ for all $\alpha'$ such that
$0\leq|\alpha'|\leq n-|\beta|$ then $r_{\beta,\alpha}=0$.
\item[(iv)] If $\alpha<_{grlex}\beta+(1,0,\ldots,0)$ and $\alpha\neq\beta$ then $r_{\beta,\alpha}=0$.
\item[(v)]$r_{\beta,\beta}=F$.
\item[(vi)]The only possibly non-zero entries of $r_{\beta}$ are those of the form
$r_{\beta,\beta+\alpha}$, for some $\alpha$ such that $0\leq|\alpha|\leq n-|\beta|$.
In particular, excepting $r_{\beta,\beta}$, these possibly non-zero entries correspond to
shifting by $\beta$ the $\alpha-$entries of $r_{(0,\ldots,0)}$, where $1\leq|\alpha|\leq n-|\beta|$,
i.e., these entries are (multiples of) the partial derivatives of $F$ of order at most $n-|\beta|$.
\end{itemize}
\end{lem}
\begin{proof}
\begin{itemize}
\item[(i)]This is just the definition of $r_{\beta}$.
\item[(ii)]Indeed, by the hypothesis,
$|\beta+\alpha|\leq n$, so it makes sense to consider $r_{\beta,\beta+\alpha}$. Now apply $(i)$.
\item[(iii)]Suppose $\alpha\neq\beta+\alpha'$ for all $\alpha'$ such that $0\leq|\alpha'|\leq n-|\beta|$.
We claim that $\alpha_i<\beta_i$ for some $1\leq i \leq s$. This is clear:
if $\alpha_i\geq\beta_i$ for all $i$, then $\alpha=\beta+(\alpha-\beta)$
and $0\leq|\alpha-\beta|=|\alpha|-|\beta|\leq n-|\beta|$, which is a contradiction.
Therefore, $(iii)$ follows from $(i)$.
\item[(iv)]This is a direct consequence of $(iii)$. Indeed, if $\alpha=\beta+\alpha'$ where $1\leq|\alpha'|$,
then $\alpha=\beta+\alpha'\geq_{grlex}\beta+(1,0,\ldots,0)$, which contradicts the hypothesis on $\alpha$.
\item[(v)]Indeed, $r_{\beta,\beta}=\frac{\partial^{\beta-\beta}F}{(\beta-\beta)!}=F$.
\item[(vi)]This is just a consequence of $(ii)$, $(iii)$, $(iv)$, and $(v)$.
\end{itemize}
\end{proof}

\begin{defi}
Let $\Jc_n(F)$ be the matrix whose rows are the $M$ vectors $r_{\beta}$. We arrange these rows
using graded lexicographical order on $\beta$, where $\beta_1<\beta_2<\ldots<\beta_s$.
In particular, $\Jc_n(F)$ is a $(M\times N-1)$-matrix. We call $\Jc_n(F)$ the
\textit{Jacobian matrix of order n} or the \textit{higher-order Jacobian matrix}.
\end{defi}

The higher-order Jacobian matrix satisfies the following properties.

\begin{pro}\label{p. higher-order jacobian}
Let $F\in\Po$, $p\in X=\V(F)\subset\C^s$, and $\bb=\langle F \rangle$.
\begin{itemize}
\item[(a)] $\Jc_1(F)$ is the usual Jacobian matrix of $F$.
\item[(b)] $\theta(\bb)=\Img(\Jc_n(F)^t_{|p})$, where $\theta$ was defined in (\ref{e. theta}) and
$\Img$ denotes the image of the linear map induced by $\Jc_n(F)^t_{|p}$.
\item[(c)] Suppose that F is a reduced non-constant polynomial. Suppose $p\in X$ is non-singular
and assume $\partial^{(1,0,\ldots,0)}F(p)\neq0$. Under this assumption, $\Jcn_{|p}$ is in
row echelon form. In addition, every row of $\Jcn_{|p}$ has $\partial^{(1,0,\ldots,0)}F(p)$ as pivot.
\end{itemize}
\end{pro}
\begin{proof}
$(a)$ is immediate by definition of $\Jcn$. $(b)$ follows from (\ref{e. theta(gF)}) and the fact
that for any $(\lambda_1,\ldots,\lambda_M)\in\C^M$ there exists $g\in\ida_p$ such that
$$\Big(\frac{\partial^{\beta}g}{\beta!}|0\leq|\beta|\leq n-1\Big)_{|p}=(\lambda_1,\ldots,\lambda_M).$$
To prove $(c)$, first notice that, for every $0\leq|\beta|\leq n-1$,
$r_{\beta,\beta+(1,0,\ldots,0)}=\partial^{(1,0,\ldots,0)}F$,
according to $(ii)$ of lemma \ref{l. remarks c-beta}. Now the fact that $\Jcn_{|p}$ is in row
echelon form follows from $(iv)$ and $(v)$ of lemma \ref{l. remarks c-beta}.
\end{proof}


\section{Higher-order criteria of non-singularity}

In this section we will give some criteria of non-singularity using the
higher-order Jacobian matrix or some other higher-order data.


\subsection{Higher-order version of the Jacobian criterion}

Our first goal is to generalize the well-known Jacobian criterion for non-singularity
(see \cite{H}, Ch. 1, Theorem 5.1). The result is the following:

\begin{teo}\label{t. higher order jacobian}
Let $F\in\Po$ be a reduced non-constant polynomial. Let $p\in X=\V(F)\subset\C^s$. For
$n\in\N$, let $M=\binom{n+s-1}{s}$. Then
$$p\mbox{ is non-singular}\Leftrightarrow rank\mbox{ }Jac_n(F)_{|p}=M.$$
\end{teo}
\begin{proof}
Suppose $p$ is non-singular and assume that $\partial^{(1,0,\ldots,0)}F(p)\neq0$.
According to $(c)$ of proposition \ref{p. higher-order jacobian}, $\Jcn_{|p}$
is in row echelon form with $\partial^{(1,0,\ldots,0)}F(p)$ as pivots in every row.
This implies that the rows of $\Jcn_{|p}$ are linearly independent, i.e.,
$\rk\mbox{ }\Jcn_{|p}=M$.
\\
\\
Suppose now that $\rk\mbox{ }\Jcn_{|p}=M$. According to $(vi)$
of lemma \ref{l. remarks c-beta}, ${r_{(0,\ldots,0,n-1)}}_{|p}$ (the last row of $\Jcn_{|p}$)
contains only first partial derivatives of $F$ as possibly non-zero entries.
If all these derivatives evaluated at $p$ were zero then $\rk\mbox{ }\Jcn_{|p}<M$.
Thus, at least one first partial derivative of $F$ evaluated at $p$ is non-zero. We conclude
that $p$ is a non-singular point by the usual Jacobian criterion.
\end{proof}

The previous theorem has the following immediate consequence.

\begin{coro}\label{c. singular set = Jac_n}
Let $\Jn\subset\Po/\langle F \rangle$ be the ideal generated by the $(M \times M)$-minors of $\Jcn$.
Then the singular locus of $X=\V(F)$ corresponds to the zero set of $\Jn$.
\end{coro}


\subsection{Some other higher-order criteria of non-singularity}

In this section we prove some other generalizations of well-known results regarding a
characterization of non-singularity. We would like to comment that we do not know if
the results of this section are particular cases of more general results. On the other
hand, the proofs given here are mostly combinatorial.

\begin{lem}\label{l. natural isom}
Let $A$ be a commutative ring with unity and $\m$ a maximal ideal of $A$. Then the natural morphism
$$\frac{\m}{\m^{n+1}}\rightarrow\frac{\m A_{\m}}{\m^{n+1}A_{\m}};\mbox{ }\mbox{ }\mbox{ }
\bar f\mapsto \Big[\frac{f}{1}\Big],$$
is an isomorphism.
\end{lem}
\begin{proof}
We proceed by induction. For $n=1$ it is well known (see, for instance, \cite{L}, Chapter 4, Lemma 2.3).
Suppose it is true for $n-1$. Consider the natural homomorphism
\begin{align}
\varphi:\frac{A}{\m^{n}}&\rightarrow\frac{A_{\m}}{\m^{n}A_{\m}},\notag\\
\bar a&\mapsto\Big[\frac{a}{1}\Big].\notag
\end{align}
Let $s\in A\setminus\m$. Then there exist $a\in A$ and $m\in\m$ such that $as+m=1$. This implies
the following equalities in $A_{\m}$:
\begin{align}
\frac{a}{1}+\frac{m}{s}&=\frac{1}{s},\notag\\
\frac{am}{1}+\frac{m^2}{s}&=\frac{m}{s},\notag\\
&\vdots\notag\\
\frac{am^{n-1}}{1}+\frac{m^n}{s}&=\frac{m^{n-1}}{s}.\notag
\end{align}
But then, modulo $\m^{n}A_{\m}$, we have:
$\big[\frac{1}{s}\big]=\big[\frac{a}{1}\big]+\big[\frac{am}{1}\big]+\cdots+\big[\frac{am^{n-1}}{1}\big]$.
Thus, $\big[\frac{b}{s}\big]=\big[\frac{ab+abm+\ldots+abm^{n-1}}{1}\big]$, which implies that $\varphi$
is surjective.
\\
\\
Now we show that $\ker\varphi\subset A/\m^{n}$ is $\{\bar0\}$. $\ker\varphi$ corresponds to some ideal
$J\subset A$ satisfying $\m^{n}\subset J\subset\m$. We want to show that $J=\m^{n}$. Suppose that there
exists $f\in J\setminus\m^{n}\subset\m\setminus\m^{n}$. This means that $\bar0\neq\bar f\in\m(A/\m^{n})$,
but $[f/1]=[0/1]\in \m(A_{\m}/\m^n A_{\m})$. On the other hand, the homomorphism $\varphi$ restricted
to $\m(A/\m^{n})$ is the natural homomorphism $\m/\m^n\rightarrow\m A_{\m}/\m^{n}A_{\m}$, which is
an isomorphism by the induction hypothesis. This contradicts that $[f/1]=[0/1]$. Therefore $\varphi$
is an isomorphism. As in \cite{L}, Chapter 4, Lemma 2.3, applying the tensor $\otimes_A\m$ we conclude
the proof of the lemma.
\end{proof}

\begin{lem}\label{l. good dim monomial ideal}
Let $\ida\subset\Po$ be a monomial ideal. Assume that there exists
$1\leq i\leq s$ such that all monomials in $\ida$ are multiples of $x_i$.
In other words, assume that $\dim V(\ida)=s-1$. Let $\m:=\langle x_1,\ldots,x_s\rangle\subset\Po/\ida$.
Then
$$\dim_{\C}\frac{\m^n}{\m^{n+1}}\geq\binom{n+s-2}{s-2}.$$
\end{lem}
\begin{proof}
Let $l=\min\{\mbox{total degree of monomials in }\ida\}$.
For every $n<l$,
$$\dim_{\C}\frac{\m^n}{\m^{n+1}}=\binom{n+s-1}{s-1}>\binom{n+s-2}{s-2},$$
so the lemma is true for these values of $n$. Now we consider $n=l+j$, $j\geq0$.
Let $L_j:=|\{x_1^{\alpha_1}\cdots x_s^{\alpha_s}\in\ida|\sum\alpha_i=l+j\}|$, where
$j\geq0$.
We claim that
\begin{align}\label{e. dim m^n/m^n+1}
\dim_{\C}\frac{\m^{l+j}}{\m^{l+j+1}}=\binom{l+j+s-1}{s-1}-L_j.
\end{align}
To prove (\ref{e. dim m^n/m^n+1}) we first observe that the ideal $\m^{l+j}$ is generated by the (classes of)
the elements of the set
$$B:=\{x_1^{\alpha_1}\cdots x_s^{\alpha_s}\notin\ida|\sum_i\alpha_i=l+j\}.$$
This set has cardinality $\binom{l+j+s-1}{s-1}-L_j$. To show that (the image of) this set is linearly
independent in $\m^{l+j}/\m^{l+j+1}$, we observe that if there were a non-trivial linear combination
of elements of $B$ equal to zero, then we would have
$\sum_{x^{\alpha}\in B}c_{\alpha}x^{\alpha}-\sum_{|\beta|=l+j+1}g_{\beta}x^{\beta}\in\ida$,
for some $c_{\alpha}\in\C$, $g_{\beta}\in\Po$, and not all of $c_{\alpha}$ equal to 0.
Thus, for some $\alpha$, $x^{\alpha}\in\ida$, since $\ida$ is a monomial ideal. This is a contradiction.
Therefore $B$ is linearly independent.
\\
\\
According to the hypothesis on $\ida$ we can assume that the variable $x_1$ appears in every
monomial of $\ida$. The set of monomials $x_1^{\alpha_1}\cdots x_s^{\alpha_s}$
of total degree $l+j$ such that $\alpha_1>0$ has cardinality $\binom{l+j+s-2}{s-1}$.
Then (\ref{e. dim m^n/m^n+1}) concludes the proof of the lemma for these values of $n$ since
$$\binom{l+j+s-1}{s-1}-L_j\geq\binom{l+j+s-1}{s-1}-\binom{l+j+s-2}{s-1}=\binom{l+j+s-2}{s-2}.$$
\end{proof}

\begin{rem}
The previous lemma is no longer valid if the hypothesis that all monomials in $\ida$ contain one
same variable $x_i$ is removed, at least for non-reduced ideals. Let $\ida=\langle x^2,y^2 \rangle$.
Then, for $n\geq3$, $\dim_{\C}\langle x,y \rangle^n/\langle x,y \rangle^{n+1}=0$.
\end{rem}

\begin{coro}\label{c. good dim iff non sing hypersurf}
Let $X=\V(F)\subset\C^s$, where $F\in\Po$. Assume that $0\in X$.
Let $\m=\langle x_1,\ldots,x_s\rangle\subset\Po/\langle F \rangle$.
Then
$$\dim_{\C}\frac{\m^n}{\m^{n+1}}\geq\binom{n+s-2}{s-2}.$$
\end{coro}
\begin{proof}
For this proof, let $\C[x]=\Po$. Consider any ideal $I\subset\C[x]$ and let
$\m=\langle x_1,\ldots,x_s \rangle\subset\C[x]/I$. Let $H_{\C[x]/I}(n):=\dim_{\C}\m^n/\m^{n+1}$
be the Hilbert function of $\C[x]/I$. Now denote by $F_0$ the homogeneous component of $F$ of
lowest degree. Let $>$ be any monomial order on $\C[x]$.
It is known that the Hilbert functions of $\C[x]/\langle F \rangle$, $\C[x]/\langle F_0 \rangle$, and that
of $\C[x]/\langle in_>(F_0) \rangle$ coincide (see \cite{E}, Theorem 15.26 and Section 15.10.3). Since
$\langle in_>(F_0) \rangle$ is an ideal generated by a single monomial, we obtain the desired conclusion
using lemma \ref{l. good dim monomial ideal}.
\end{proof}

\begin{coro}\label{c. non-sing iff good dim at order n}
Let $X=\V(F)\subset\C^s$, where $F\in\Po$. Let $p\in X$ and let $\m_p$ be its corresponding maximal ideal in
$\Ox_{X,p}$. Then $p$ is non-singular if and only if $\dim_{\C}\m_p/\m_p^{n+1}=\binom{n+s-1}{s-1}-1$.
\end{coro}
\begin{proof}
By lemma \ref{l. natural isom}, it is enough to prove the statement for $\m/\m^{n+1}$, where $\m$ is the
maximal ideal corresponding to $p$ in $\Po/\langle F\rangle$.
We proceed by induction on $n$. Let $n=1$. If $p$ is non-singular,
$\dim_{\C}(\m/\m^2)=\dim X=s-1=\binom{s}{s-1}-1$.
Now consider the exact sequence of $\C-$vector spaces:
\begin{equation}\label{e. exact sequence}
0\rightarrow\frac{\m^n}{\m^{n+1}}
\rightarrow \frac{\m}{\m^{n+1}}\rightarrow\frac{\m}{\m^n}\rightarrow0.
\end{equation}
Since $p$ is non-singular, $\textbf{S}^n(\m/\m^2)=\m^n/\m^{n+1}$, where $\textbf{S}^n(\cdot)$
denotes the $n$th-symmetric product. Thus $\dim_{\C}(\m^n/\m^{n+1})=\binom{n+s-2}{s-2}$. On the other hand,
by induction, $\dim_{\C}(\m/\m^n)=\binom{n+s-2}{s-1}-1$. By exactness of the sequence, we conclude that
$$\dim_{\C}\Big(\frac{\m}{\m^{n+1}}\Big)=\binom{n+s-2}{s-2}+\binom{n+s-2}{s-1}-1=\binom{n+s-1}{s-1}-1.$$
Now suppose that $p\in X$ is singular. In particular, $\dim_{\C}\m/\m^2>\dim X=s-1$. Using corollary
\ref{c. good dim iff non sing hypersurf} and the exact sequence (\ref{e. exact sequence}),
we conclude by induction that $\dim_{\C}\m/\m^{n+1}>\binom{n+s-1}{s-1}-1$.
\end{proof}


\section{Computing $T^n_pX$}

Let $F\in\Po$, $R=\Po/\langle F \rangle$ and $X=\V(F)$.
Let $p\in X$ and $(R_p,\m_p)$ be the localization of $R$ at $p$.
It is well known that the tangent space at $p$ has the following description:
$$T_pX=(\m_p/\m_p^{2})^{\vee}=\ker\Jc(F)_{|p},$$
where $\Jc(F)$ is the Jacobian matrix of $F$. The goal of this section is to give an analogous
description of the $R_p/\m_p\cong\C-$vector space
$$T^n_pX=(\m_p/\m_p^{n+1})^{\vee},$$
for non-singular points using the higher-order Jacobian. As in previous sections, for $n\in\N$,
let $N=\binom{n+s}{s}$ and $M=\binom{n+s-1}{s}$.

\begin{rem}
Notice that for non-singular points of a hypersurface
$p\in X\subset\C^s$, $T_pX$ is a hyperplane in $\C^s$. However, for $n>1$ the space
$T^n_pX\hookrightarrow T^n_p\C^s\cong\C^{N-1}$ is not a hyperplane of $\C^{N-1}$
(see corollary \ref{c. non-sing iff good dim at order n}).
\end{rem}

\begin{lem}\label{l. TnX vs dual}
Let $p\in X\subset\C^s$. Fix $n\in\N$ and let $\Jcn$ be the higher-order
Jacobian matrix of $F$. Then we have the following identification:
$$T^n_pX\cong\Big(\frac{\C^{N-1}}{\Img(\Jc_n(F)^t_{|p})}\Big)^{\vee}.$$
\end{lem}
\begin{proof}
The proof consists in adapting the usual proof for the case $n=1$ to the case $n>1$
(see for instance \cite{H}, Chapter I, Theorem 5.1). Let $p=(a_1,\ldots,a_s)$ and
$\ida_p=\langle x_1-a_1,\ldots,x_s-a_s\rangle\subset\Po$.
As in section \ref{s. higher-order jacobian}, consider the following linear map:
\begin{align}
\theta:\ap&\rightarrow\C^{N-1}\notag\\
f&\mapsto\Big(\frac{\partial^{\alpha}f}{\alpha!}|1\leq|\alpha|\leq n\Big)_{|p}.\notag
\end{align}
This map is surjective since $(x-a)^{\alpha}$, for $1\leq|\alpha|\leq n$, are mapped to
the canonical basis of $\C^{N-1}$. On the other hand, by observing the Taylor expansion of
an element of $\ap$ around $p$ we see that $\ker\theta=\ap^{n+1}$. Thus,
\begin{equation}\label{e. isom all n}
\frac{\ap}{\ap^{n+1}}\cong\C^{N-1}.
\end{equation}
Let $\bb=\langle F \rangle$. According to $(b)$ of proposition \ref{p. higher-order jacobian},
$\theta(\bb)=\Img(\Jc_n(F)^t_{|p})$. Using the isomorphism (\ref{e. isom all n}) we also have
$\theta(\bb)\cong\bb(\ap/\ap^{n+1})=(\bb+\ap^{n+1})/\ap^{n+1}$. Then
$$\frac{\frac{\ap}{\ap^{n+1}}}{\frac{\bb+\ap^{n+1}}{\ap^{n+1}}}\cong\frac{\ap}{\bb+\ap^{n+1}}.$$
Since $\bb\subset\ap$, it follows that
$$\frac{\ap\Big(\frac{\Po}{\bb}\Big)}{\ap^{n+1}\Big(\frac{\Po}{\bb}\Big)}
\cong\frac{\frac{\ap+\bb}{\bb}}{\frac{\ap^{n+1}+\bb}{\bb}}=\frac{\frac{\ap}{\bb}}{\frac{\ap^{n+1}+\bb}{\bb}}
\cong\frac{\ap}{\ap^{n+1}+\bb}.$$
By identifying (see lemma \ref{l. natural isom}) $\frac{\m_p}{\m_p^{n+1}}\cong\frac{\ap\Big(\frac{\Po}{\bb}\Big)}{\ap^{n+1}\Big(\frac{\Po}{\bb}\Big)},$
we conclude
$$T^n_pX=\Big(\frac{\m_p}{\m_p^{n+1}}\Big)^{\vee}\cong
\Big(\frac{\frac{\ap}{\ap^{n+1}}}{\frac{\bb+\ap^{n+1}}{\ap^{n+1}}}\Big)^{\vee}
\cong\Big(\frac{\C^{N-1}}{\Img(\Jc_n(F)^t_{|p})}\Big)^{\vee}.$$
\end{proof}

Now assume that $p\in X$ is a non-singular point. We claim that
\begin{equation}\label{e. ker equals dual}
\Big(\frac{\C^{N-1}}{\Img(\Jc_n(F)^t_{|p})}\Big)^{\vee}\cong\ker\Jc_n(F)_{|p}.
\end{equation}
To prove this, let $e_i^{\vee}:\C^{N-1}\rightarrow\C$, $t=(t_1,\ldots,t_{N-1})\mapsto t_i$. Then
\begin{align}\label{e. ker in dual}
\ker\Jc_n(F)_{|p}\hookrightarrow\{\phi:\C^{N-1}\rightarrow\C|\Img(\Jc_n(F)^t_{|p})\subset\ker\phi\}
\end{align}
via the map $t\mapsto\sum_it_ie_i^{\vee}$. Since $p$ is a non-singular point,
theorem \ref{t. higher order jacobian} implies
$\dim\ker\Jc_n(F)_{|p}=N-M-1.$ Now, since $(\C^{N-1}/\Img(\Jc_n(F)^t_{|p}))^{\vee}\cong
\{\phi:\C^{N-1}\rightarrow\C|\Img(\Jc_n(F)^t_{|p})\subset\ker\phi\}$,
corollary \ref{c. non-sing iff good dim at order n} and lemma \ref{l. TnX vs dual} imply:
$\dim\{\phi:\C^{N-1}\rightarrow\Img(\Jc_n(F)^t_{|p})\subset\ker\phi\}
=\dim_{\C}(\m_p/\m_p^{n+1})^{\vee}=N-M-1.$
Therefore the inclusion (\ref{e. ker in dual}) is actually an equality. This proves claim
(\ref{e. ker equals dual}). Using lemma \ref{l. TnX vs dual} we conclude:

\begin{pro}
Let $F\in\Po$ and $p\in X=\V(F)\subset\C^s$ be a non-singular point. Then
$T^n_pX=\ker\Jc_n(F)_{|p}.$
\end{pro}


\section{Some applications}

In this final section we will give some applications of the constructions and results of previous
sections. Firstly, we will generalize a result of O'Shea appearing in \cite{Sh} that computes limits
of tangent spaces to singular points of a hypersurface. Secondly, applying some results of Villamayor
appearing in \cite{V}, we will describe an ideal whose blow up is the higher
Nash blowup of a hypersurface. Using this ideal, we will prove a higher-order analogue of Nobile's
theorem for normal hypersurfaces.


\subsection{Limits of $T^n_pX$, where $X$ is a hypersurface}

We start by revisiting a theorem due to D. O'Shea appearing in \cite{Sh} which gives a method to
compute limits of tangent spaces to a singular point of a hypersurface. We will see that this result
is still valid if we replace tangent space by $T_p^nX$, for any $n\in\N$, essentially with the same proof.
This theorem will allow us to compute the space of limits of $T_p^nX$ using the theory of Gr\"obner bases.
In particular, this method provides a way to compute the fibers of the higher Nash blowup of a hypersurface.

\begin{defi}
Let $X=\V(F)\subset\C^s$ where $F\in\Po$, and let $S(X)$ denotes the singular locus of $X$.
Assume that $0\in X$. The space of limits of $T_p^nX$ at $0$ is the set
$$\{T\in Gr(N-M-1,\C^{N-1})|\exists\{p_k\}\subset X\setminus S(X) \mbox{ s.t. }
p_k\rightarrow 0\mbox{ and }T_{p_k}^nX\rightarrow T\},$$
where $Gr(N-M-1,\C^{N-1})$ denotes the Grassmanian of vector spaces of dimension $N-M-1$ in $\C^{N-1}$.
We denote the space of limits of $T_p^nX$ as $\Li$. By using Pl\"ucker coordinates, we embed
$Gr(N-M-1,\C^{N-1})$ in a projective space so, when we mention the space $T_p^nX$ or a limit of such,
we consider them as points in such a projective space.
\end{defi}

\begin{rem}
We will use the duality between $Gr(N-M-1,\C^{N-1})$ and $Gr(M,\C^{N-1})$ to compute $\Li$.
More precisely, in the next theorem we will compute limits of vector spaces of dimension $M$
defined as the span of the rows of $\Jcn_{|p_k}$, where $\{p_k\}\subset X$ is a sequence of non-singular
points converging to 0 (recall that $\rk(\Jcn_{|p_k})=M$ in this case).
By duality, we will obtain the set $\Li$.
\end{rem}

Let $\Lambda=\{(\alpha_1,\ldots,\alpha_M)|\alpha_i\in\N^s,\mbox{ }
1\leq|\alpha|\leq n,\mbox{ }\alpha_1<_{grlex}\ldots<_{grlex}\alpha_M\}$.
For $J\in\Lambda$, denote by $\Delta_J$ the determinant of the matrix formed by the $M$
columns of $\Jcn$ corresponding to $J$.

\begin{teo}\label{t. L(V,0) = V(A)}(cf. \cite{Sh}, Proposition 1)
Let $X=\V(F)\subset \C^s$ be a hypersurface, where $F\in\C[x_1,\ldots,x_s]$ and
assume that $0\in X$ is a singular point.
Consider the following ideal in $\C[x_1,\ldots,x_s,t,u_J|J\in\Lambda]$:
$$A=\langle F,u_J-t\Delta_J|J\in\Lambda\rangle.$$
Then $\Li$ can be identified with the variety
$$\V\Big(\frac{A\cap\C[x_1,\ldots,x_s,u_J|J\in\Lambda]}{\langle x_1,\ldots,x_s\rangle}\Big).$$
\end{teo}
\begin{proof}
Make $(x,u)=((x_1,\ldots,x_s),(u_J|J\in\Lambda))$. The idea of the proof consists in showing that points in
$\V(A\cap\C[x,u])$ represent points $x\in X$ along with complex multiples of $(\Delta_J(x)|J\in\Lambda)$ or
limits of such.
\\
\\
Suppose first that $x_0\in X$ is non-singular. We claim that $(x_0,u_0)\in\V(A\cap\C[x,u])$
if and only if $u_0$ is a complex multiple of $(\Delta_J(x_0)|J\in\Lambda)$.
Let $(x_0,u_0)\in\V(A\cap\C[x,u])$. Since $x_0$ is non-singular, by theorem \ref{t. higher order jacobian},
$\Delta_J(x_0)\neq0$ for some $J\in\Lambda$. In particular, $(x_0,u_0)\notin \V(F,\Delta_J|J\in\Lambda)$.
According to \cite{CLO}, Ch. 3, Section 1, Theorem 3, the partial solution $(x_0,u_0)\in\V(A\cap\C[x,u])$
extends to a solution $(x_0,t_0,u_0)\in\V(A)$, for some $t_0\in\C$. This implies that
${u_0}_J-t_0\Delta_J(x_0)=0$ for all $J\in\Lambda$, i.e., $u_0=t_0(\Delta_J(x_0)|J\in\Lambda).$
Suppose now that $u_0$ is a complex multiple of $(\Delta_J(x_0)|J\in\Lambda)$.
In particular, $(x_0,t,u_0)\in\V(A)$, for some $t\in\C$. This immediately implies that
$(x_0,u_0)\in\V(A\cap\C[x,u])$ (see \cite{CLO}, Ch. 3, Section 2, Lemma 1).
\\
\\
Now we suppose that $x_0\in X$ is singular. We claim that $(x_0,u_0)\in\V(A\cap\C[x,u])$ if and
only if $u_0$ is limit of multiples of $(\Delta_J(x)|J\in\Lambda)$ for a sequence of non-singular
points converging to $x_0$.
\\
\\
We start with the second implication. Let $\{x_k\}\subset X\setminus S(X)$ be a sequence
such that $x_k\to x_0$ and let $\{u_k\}$ be the sequence of
multiples of $(\Delta_J(x_k)|J\in\Lambda)$ converging to $u_0$. By the non-singular case
we have that $(x_k,u_k)\in\V(A\cap\C[x,u])$ for all $k$. Then $(x_k,u_k)\rightarrow(x_0,u_0)$
and since $\V(A\cap\C[x,u])$ is a closed set then we must have $(x_0,u_0)\in \V(A\cap\C[x,u])$.
\\
\\
Let us suppose now that $(x_0,u_0)\in\V(A\cap\C[x,u])$. Since $x_0$ is singular, $\Delta_J(x_0)=0$
for all $J\in\Lambda$, according to theorem \ref{t. higher order jacobian}. On the other hand,
we can assume $u_0\neq 0$ (if $u_0=0$ the claim is trivially true).
These facts imply $(x_0,t,u_0)\notin\V(A)$ for all $t\in \C$. Therefore,
$(x_0,u_0)\notin\pi_t(\V(A))$, where $\pi_t$ is the projection to the $x$ and $u$ coordinates.
According to \cite{CLO}, Ch. 3, Section 2, Theorem 3, we know that
$\overline{\pi_t(\V(A))}=\V(A\cap\C[x,u])$.
Since $(x_0,u_0)\in\V(A\cap\C[x,u])$ it follows that $(x_0,u_0)$ is limit of points in $\pi_t(\V(A))$
(notice that we are using the fact that topological and algebraic closure coincides), i.e.,
$(x_k,u_k)\to(x_0,u_0)$ for some sequence $\{(x_k,u_k)\}\subset\pi_t(\V(A))$. Thus, there exists
$\{(x_k,t_k,u_k)\}\subset\V(A)$ such that $\pi_t(x_k,t_k,u_k)=(x_k,u_k)$ In particular, $u_k$ is
a complex multiple of $(\Delta_J(x_k)|J\in\Lambda)$.
If $x_k\in X$ is singular for all $k$ then $u_k=0$, so that $(x_k,0)=\pi_t(x_k,t_k,0)$ is such that
$(x_k,0)\to(x_0,u_0)\neq(x_0,0)$, which is a contradiction. We conclude that there are at most a finite
number of singular points in $\{x_k\}$. Taking $k$ sufficiently large, we have a non-singular sequence.
This finishes the proof of the claim.
\\
\\
To conclude the proof of the theorem we notice the following natural bijective correspondence:
$\V(A\cap\C[x,u])\cap\{x=0\}\leftrightarrow\V((A\cap\C[x,u])/\langle x \rangle).$
Thus, points in $\V((A\cap\C[x,u])/\langle x \rangle)$ correspond to limits of complex multiples
of vectors $(\Delta_J(x_k)|J\in\Lambda)$, where $\{x_k\}\subset X\setminus S(X)$ and $x_k\rightarrow0$.
These vectors determine the spaces $T_{x_k}^nX$. We have obtained the desired identification.
\end{proof}

Next we present a simple example to illustrate the method of the previous theorem.

\begin{exam}
Let $F=x^3-y^2$ and $X=\V(F)\subset\C^2$. After computing the corresponding minors of $\Jcd$
(see (\ref{e. Jac2(x3-y2)})) we define:
\begin{align}
A=\langle F, &u_1-(3xF-9x^4)tF,u_2-(12x^2y)tF,u_3+(4y^2+F)tF,\notag\\
&u_4-27x^6t+(9x^3)tF,u_5+18x^4yt+(6xy)tF,u_6-12x^2y^2t-(3x^2)tF,\notag\\
&u_7+18x^4yt-(6xy)tF,u_8-12x^2y^2t+(3x^2)tF,u_9+8y^3t+(2y)tF,\notag\\
&u_{10}-(12xy^2-9x^4)t\rangle.\notag
\end{align}
$A$ is an ideal in $\C[x,y,t,u_1,\ldots,u_{10}]$. Now we use the theory of Gr\"obner bases
to compute a basis of $A\cap\C[x,y,u_1,\ldots,u_{10}]/\langle x,y \rangle$. First, using
$\mathtt{SINGULAR}$ $\mathtt{3}$-$\mathtt{1}$-$\mathtt{6}$ (\cite{DGPS}),
we compute a Gr\"obner basis of $A$ with respect to lexicographical order assuming $t>x>y>u_i$,
call it $G$. Then $G\cap\C[x,y,u_1,\ldots,u_{10}]$ is a basis of $A\cap\C[x,y,u_1,\ldots,u_{10}]$
(see \cite{CLO}, Ch. 3, Section 1, Theorem 2). By making $x=0$, $y=0$ in the resulting set we
obtain $A\cap\C[x,y,u_1,\ldots,u_{10}]/\langle x,y \rangle=
\langle u_1,u_2,u_3,u_4,u_5,u_6^2,u_8,u_9,u_{10} \rangle.$
It follows that the zero set of this ideal is $L=\{(0,\ldots,0,a_7,0,0,0)\in\C^{10}\}$.
This means that $\Li$ consists of only one limit of spaces $T^2_pX$, for any sequence of
non-singular points converging to the origin, which corresponds to the projectivization of $L$.
\end{exam}

\begin{exam}
Let $F=x^3+x^2-y^2$ and $X=\V(F)\subset\C^2$. Consider the Jacobian matrix of order 2 of $F$:
\[
\Jc_2(F):=
\begin{pmatrix}
3x^2+2x& 	-2y& 	3x+1&  	     0&     -1\\
F&  	      0&    3x^2+2x&   -2y&      0\\
0&            F&     0&      3x^2+2x&   -2y
\end{pmatrix}
\]
As in the previous example we find that a basis of $A\cap\C[x,y,u_1,\ldots,u_{10}]/\langle x,y \rangle$
is given by the set $\{u_1-2u_7,u_2-u_3,u_3-u_6,u_4-u_5,u_5-2u_7,u_6^2-4u_7^2,u_8,u_9,u_{10}\}$.
The zero set of this ideal is the following set:
\begin{align}
L=\{(a_1,a_2,\ldots,a_{10})\in\C^{10}|a_1=a_4=a_5=2a_7,a_2=a_3=a_6,&\notag\\
a_2^2-4a_7^2=0,a_8=a_9=a_{10}=0\}.\notag
\end{align}
In particular, there are only two different limits of spaces $T^2_pX$ corresponding to the
projectivization of $L$.
\end{exam}

The higher Nash blowup is a modification of a variety. In particular, for curves, its fibers are finite
sets. This is not necessarily true for varieties of higher dimension as the following example shows.

\begin{exam}
Let $F=xy-z^4$ and $X=\V(F)\subset\C^3$. It is well known that $\mathcal L_1(X,0)$ is an infinite set:
any plane in $\C^3$ containing the $z$-axis is a limit of tangent spaces (see the example following Proposition
1 in \cite{Sh}). Now we show that $\mathcal L_2(X,0)$ is also infinite.
Consider the Jacobian matrix of order 2 of $F$:
\[
\Jc_2(F):=
\begin{pmatrix}
y&       	  x& 	-4z^3&       0&      1&     0&      0&      0&      -6z^2\\
F&  	      0&        0&       y&      x&     0&    -4z^3&    0&        0\\
0&  	      F&        0&       0&      y&     x&      0&    -4z^3&      0\\
0&  	      0&        F&       0&      0&     0&      y&      x&      -4z^3
\end{pmatrix}
\]
After carefully computing the minors of $\Jc_2(F)$, a basis for the ideal
$A\cap\C[x,y,z,u_1,\ldots,u_{126}]/\langle x,y,z \rangle$ is given by the following set:
\begin{align}
\{&u_1,\ldots,u_{37},u_{39},\ldots,u_{82},u_{84},\ldots,u_{111},u_{114}^2,u_{115}^3,u_{116}^2,
u_{117},\ldots,u_{121},\notag\\
&u_{122}^2,u_{123},u_{124}^2,u_{125},u_{126},u_{38}u_{83},u_{38}u_{113},u_{38}u_{114},u_{38}u_{122},u_{38}u_{124},\notag\\
&u_{83}u_{112},u_{83}u_{114},u_{83}u_{115},u_{83}u_{116},u_{112}u_{124},u_{113}u_{115}^2,u_{113}u_{116},\notag\\
&u_{114}u_{115},u_{114}u_{116},u_{114}u_{122},u_{114}u_{124},u_{115}u_{116},u_{115}u_{122},u_{115}u_{124},\notag\\
&u_{116}u_{122},u_{116}u_{124},u_{122}u_{124},u_{112}u_{114}+u_{113}u_{115},u_{112}u_{122}-u_{113}u_{114},\notag\\
&8u_{112}u_{116}+3u_{115}^2,8u_{113}u_{124}+3u_{122}^2\}\notag.
\end{align}
The zero set of this ideal in $\C^{126}$ is:
\begin{align}
L=\{(a_1,a_2,\ldots,a_{126})\in\C^{126}|&a_1=\cdots=a_{37}=a_{39}=\cdots=a_{82}=a_{84}=0,\notag\\
&a_{85}=\cdots=a_{111}=a_{114}=\cdots=a_{126}=0,\notag\\
&a_{38}a_{83}=a_{38}a_{113}=a_{83}a_{112}=0\}\notag.
\end{align}
$L$ corresponds to three 2-dimensional planes in $\C^{126}$: $P_1=\mbox{span}\{e_{112},e_{113}\}$,
$P_2=\mbox{span}\{e_{38},e_{112}\}$, $P_3=\mbox{span}\{e_{83},e_{113}\}$. After projectivization we obtain
three lines in $\mathbb{P}^{125}$, call them $l_1$, $l_2$, $l_3$. These $l_i$ give place to the following
families of 4-dimensional vector spaces of $\C^{9}$:
$$\{(0,0,0,\lambda_4,\lambda_5,\lambda_6,\lambda_7,\lambda_8,0)\in\C^9|a\lambda_8-b\lambda_7=0,
(a,b)\in\C^2\setminus\{(0,0)\}\},$$
$$\{(\lambda_1,0,0,\lambda_4,\lambda_5,\lambda_6,\lambda_7,0,0)\in\C^9|a\lambda_1-d\lambda_6=0,
(a,d)\in\C^2\setminus\{(0,0)\}\},$$
$$\{(0,\lambda_2,0,\lambda_4,\lambda_5,\lambda_6,0,\lambda_8,0)\in\C^9|b\lambda_2-c\lambda_4=0,
(b,c)\in\C^2\setminus\{(0,0)\}\},$$
respectively. Taking orthogonal complements, we have that any 5-dimensional vector space $W\subset\C^9$ such that
$W\subset\mbox{span}(e_1,e_2,e_3,e_7,e_8,e_9)$ and contains $\mbox{span}(e_1,e_2,e_3,e_9)$, or
$W\subset\mbox{span}(e_1,e_2,e_3,e_6,e_8,e_9)$ and contains $\mbox{span}(e_2,e_3,e_8,e_9)$, or
$W\subset\mbox{span}(e_1,e_2,e_3,e_4,e_7,e_9)$ and contains $\mbox{span}(e_1,e_3,e_7,e_9)$, is
a limit of spaces $T_p^2X$, where $e_i$ denotes the canonical basis of $\C^9$.
\end{exam}


\subsection{An ideal defining the higher Nash blowup of a hypersurface}\label{s. ideal defining higher Nash}

The goal of this section is to prove that the ideal whose blowup is the higher Nash
blowup of order $n$ of a hypersurface, correspond to the ideal generated by the maximal
minors of the Jacobian matrix of order $n$ of the polynomial defining the hypersurface.
This fact will be a direct consequence of a more general result of O. Villamayor appearing
in \cite{V}. With this ideal at hand, we prove that the higher-order version of Nobile's theorem
is true for normal hypersurfaces. We also exhibit some examples of singular plane curves where
this result holds for the higher Nash blowup of order 2.
\\
\\
Let us first expose the results we are going to need. Let $F\in\Po$ be an irreducible polynomial,
$R=\Po/\langle F \rangle$ and $X=\Sp R$. Let $I:=\ker(R\otimes R\rightarrow R)$. We give structure of
$R-$module to $I$ via the map $R\rightarrow R\otimes R$, $r\mapsto r\otimes1$. Let $K$ be the field of
fractions of $R$ and let $r=\dim_K I/I^{n+1}\otimes_R K$ be the generic rank of $I/I^{n+1}$. Consider
the following fractionary ideal of $K$:
$$\mathfrak b:=Im(\bigwedge^{r}\frac{I}{I^{n+1}}\rightarrow\bigwedge^{r}\frac{I}{I^{n+1}}\otimes_R K\cong K).$$

\begin{teo}
The higher Nash blowup of $X$ is isomorphic to the blowup of the fractionary ideal $\mathfrak b$.
\end{teo}
\begin{proof}
See \cite{Y}, Proposition 1.8, and \cite{OZ}, Theorem 3.1.
\end{proof}

The ideal $\mathfrak b$ can be explicitly described as follows. Consider a presentation of the module $I/I^{n+1}$
by a $(\Lambda\times\Lambda')$-matrix $A$:
\begin{equation}\label{e. presentation}
\xymatrix{R^{\Lambda'}\ar[r]^{A}&R^{\Lambda}\ar[r]&\frac{I}{I^{n+1}}\ar[r]&0}.
\end{equation}
Then there exist $\Lambda-r$ columns of $A$ such that the $\Lambda\times(\Lambda-r)$-matrix $A'$ formed by these
columns has rank $\Lambda-r$.

\begin{pro}\label{p. blown up ideal}
The ideal $\Jn\subset R$ generated by the $(\Lambda-r)$-minors of $A'$ is equal to $\mathfrak b$ for
a suitable choice of isomorphism $\bigwedge^{r}\frac{I}{I^{n+1}}\otimes_R K\cong K$. In addition, the ideal
$\Jn$ is independent of the choice of the $\Lambda-r$ columns of $A$ as long as the rank of the matrix
formed by these columns is $\Lambda-r$.
\end{pro}
\begin{proof}
This is a particular case of \cite{V}, Proposition 2.5 and Corollary 2.6. For $n=1$, this is Theorem 1 of
\cite{No} or Theorem 1 of \cite{GS-1} (Section 2).
\end{proof}

With these results at hand, now we can look for the ideal defining the higher Nash blowup.
It is well known that $I=\langle x_i\otimes 1-1\otimes x_i|i=1,\ldots,s\rangle$.
Now consider the following isomorphisms of rings (let $x=(x_1,\ldots,x_s)$, $x'=(x_1',\ldots,x_s')$):
\begin{align}
R\otimes_{\C} R&\cong\C[x,x']/\langle F(x),F(x')\rangle\cong\C[x,x'-x]/\langle F(x),F(x')\rangle\notag\\
          &(\mbox{let }\Delta x:=x'-x)\notag\\
          &\cong\C[x,\Delta x]/
          \langle F(x),F(x+\Delta x)\rangle\notag\\
          &=\C[x,\Delta x]/\langle F(x),\sum_{|\alpha|\geq1}\frac{\partial^{\alpha}F}{\alpha!}
          (\Delta x)^{\alpha}\rangle.\notag
\end{align}
In this isomorphic ring, $I=\langle\Delta x_1,\ldots,\Delta x_s\rangle$.
Thus, the quotient of $R-$modules $I/I^{n+1}$ is generated by
$\{[(\Delta x)^{\alpha}]|1\leq|\alpha|\leq n\}.$ This set has cardinality $N-1$
(recall that $N=\binom{n+s}{s}$).
\\
\\
Let $\{e_{\alpha}|1\leq|\alpha|\leq n\}$ denotes the canonical basis of $R^{N-1}$ (we arrange the set of
such $\alpha$ increasingly by graded lexicographical order assuming $\alpha_1<\ldots<\alpha_s$).
Consider the following surjective map, $\theta:R^{N-1}\rightarrow I/I^{n+1}$,
$e_{\alpha}\mapsto[(\Delta x)^{\alpha}]$. Viewing the rows $r_{\beta}$ of $\Jcn$
as elements of $R^{N-1}$ (so the entries of $\Jcn$ are taken modulo $F$),
we notice that (see ($i$) of lemma \ref{l. remarks c-beta}):
\begin{align}\label{e. relations I/I^n+1}
\theta(r_{\beta})&=\Big[\sum_{1\leq|\alpha|\leq n}\Big(\beta!\binom{\alpha}{\beta}
\frac{\partial^{\alpha-\beta}F}{\alpha!}\Big)(\Delta x)^{\alpha}\Big]\notag\\
&=\Big[\sum_{1\leq|\alpha|\leq n,\mbox{ }\alpha>\beta}
\frac{\partial^{\alpha-\beta}F}{(\alpha-\beta)!}(\Delta x)^{\alpha}\Big]\notag\\
&=\Big[(\Delta x)^{\beta}\Big(\sum_{1\leq|\alpha|\leq n,\mbox{ }\alpha>\beta}
\frac{\partial^{\alpha-\beta}F}{(\alpha-\beta)!}(\Delta x)^{\alpha-\beta}\Big)\Big]\notag\\
&=\Big[(\Delta x)^{\beta}\Big(-\sum_{|\alpha|>n-|\beta|}
\frac{\partial^{\alpha}F}{\alpha!}(\Delta x)^{\alpha}\Big)\Big]=[0],
\end{align}
where the last equality follows from the fact that every element $(\Delta x)^{\alpha+\beta}$
appearing on the sum satisfies $|\alpha|+|\beta|>n-|\beta|+|\beta|=n$. In particular,
every row of $\Jcn$ represents a relation of the generators of $I/I^{n+1}$.

\begin{lem}\label{l. generic rank}
The generic rank of $I/I^{n+1}$ is $\binom{n+s-1}{s-1}-1$.
\end{lem}
\begin{proof}
This is just the local version of known results on the sheaf of principal parts
(see \cite{G}, Paragraph 16.3.7 and \cite{LT}, Section 4).
\end{proof}

\begin{pro}\label{p. higher Nash = blowup minors}
The ideal $\Jn$ defining the higher Nash blowup of $X$ coincides with the ideal generated by
the maximal minors of $\Jcn$.
\end{pro}
\begin{proof}
Consider the following presentation of $I/I^{n+1}$:
$$\xymatrix{R^{\Lambda'}\ar[r]^{A}&R^{N-1}\ar[r]^{\theta}&\frac{I}{I^{n+1}}\ar[r]&0}.$$
According to (\ref{e. relations I/I^n+1}), we have that $\Jcn^t$ is a submatrix of $A$.
After change of coordinates if necessary, we can assume that $\partial^{(1,0,\ldots,0)}F\neq0$.
As in proposition \ref{p. higher-order jacobian}, it follows that $\Jcn$ is in row echelon form.
Thus $\rk\mbox{ }\Jcn=M$ (recall that $M=\binom{n+s-1}{s}$). On the other hand, since
$M=N-1-(\binom{n+s-1}{s-1}-1)$, the previous lemma implies that $\Jcn^t$ satisfies the requirements
on the submatrix $A'$ of $A$ in (\ref{e. presentation}). By proposition \ref{p. blown up ideal},
we conclude that the ideal whose blowup gives the higher Nash blowup of $X$, coincides with the ideal
generated by the maximal minors of $\Jcn$.
\end{proof}

To give an example of how can we use the explicit description of the ideal defining the
higher-order Nash blowup of a hypersurface, we are going to study a higher-order
version of the following theorem due to A. Nobile.

\begin{teo}
Let $X$ be an equidimensional algebraic variety over $\C$. Let $(X^*,\nu)$ be the Nash
modification of $X$. Then $\nu$ is an isomorphism if and only if $X$ is non-singular.
\end{teo}
\begin{proof}
See \cite{No}, Theorem 2.
\end{proof}

We can naturally ask if this theorem holds when we replace Nash modification by the
higher Nash blowup. The next example considers the case of the second Nash
blowup of some singular plane curves.

\begin{exam}
Let $F=y^p-x^q\in\C[x,y]$, where $2\leq p < q$ and $(p,q)=1$. Let $X=\V(F)\subset\C^2$. After computing
the maximal minors of $\Jc_2(F)$, we obtain that the ideal $\mathcal{J}_2$ of proposition \ref{p. higher Nash = blowup minors}
is generated by (recall that we take the minors modulo $F$):
\[\mathcal{J}_2=
\left\{
\begin{array}{rll}
&\langle x^{q-2}y^{2p-2}, y^{3p-3} \rangle, & \mbox{ if } p=2,3,\\
&\langle x^{q-3}y^{2p},x^{q-2}y^{2p-2},y^{3p-3} \rangle, & \mbox{ if } p>3.
\end{array}
\right.
\]
Notice that $\mathcal{J}_2$ is a non-principal ideal in every case (this can be seen by using
the isomorphism $\C[x,y]/\langle F \rangle\cong\C[u^p,u^q]$ and the hypothesis on $p$, $q$).
In particular, the higher Nash blowup of order 2 of $X$ is not an isomorphism
(see \cite{L}, Chapter 8, Proposition 1.12).
Since $X$ is a singular curve, this shows that the analogue of Nobile's theorem on the usual
Nash blowup is also true for the higher Nash blowup of order 2 for this family of curves.
\end{exam}

Even though the strategy in the previous example is unlikely to work for the general case of a
hypersurface, we can still use proposition \ref{p. higher Nash = blowup minors} to show that
the higher-order analogue of Nobile's theorem holds for normal hypersurfaces.

\begin{teo}
Let $F\in\Po$ be an irreducible polynomial and $X=\V(F)\subset\C^s$. Suppose $X$ is normal.
Let $(\Na,\pi_n)$ be the higher Nash blowup of order $n$ of $X$. Then $\pi_n$ is an
isomorphism if and only if $X$ is non-singular.
\end{teo}
\begin{proof}
$\pi_n$ only modifies singular points so if $X$ is non-singular then $\pi_n$ is an isomorphism.
Suppose now that $X$ is singular. Let $\Jn$ be the ideal defining the higher Nash blowup of
order $n$ of $X$. According to proposition \ref{p. higher Nash = blowup minors} and corollary
\ref{c. singular set = Jac_n}, the zero set of $\Jn$ coincides with the singular locus
$S(X)$ of $X$. Since $X$ is normal, $\dim S(X)\leq d-2$, where $d=\dim X$. It follows that $\Jn$
must be generated by at least two elements. Now we use the fact that the blowup of a non-principal
ideal is not an isomorphism.
\end{proof}

\section*{Acknowledgements}

I want to thank Takehiko Yasuda for his great help during the preparation of this note.
He read previous versions of this paper and made valuable comments. In particular, it was
he who noticed that the higher-order version of Nobile's theorem for normal hypersurfaces
was a consequence of theorem \ref{t. higher order jacobian} and proposition
\ref{p. higher Nash = blowup minors}. I also thank Mark Spivakovsky from whom I
learned some of the computations appearing in section \ref{s. ideal defining higher Nash}.
Those computations were the original motivation of the present work. Finally, I thank
Jawad Snoussi for having introduced me to the reference \cite{Sh}.

\vspace{0.5cm}
{\footnotesize\textsc{Unidad Acad\'emica de Matem\'aticas de la Universidad Aut\'onoma de Zacatecas, a trav\'es
del programa C\'atedras para J\'ovenes Investigadores (Conacyt). Calzada Solidaridad y Paseo La Bufa,
Col. Hidr\'aulica, C.P. 98060, Zacatecas, Zacatecas.}}
\\
{\footnotesize\textsc{Email}: adduarte@matematicas.reduaz.mx}

\begin{thebibliography}{XXX}
\bibitem[DGPS]{DGPS}W. Decker, G.-M. Greuel, G. Pfister, H. Sch\"onemann;
\newblock {\sc Singular} {3-1-6} --- {A} computer algebra system for polynomial computations.
\newblock {http://www.singular.uni-kl.de} (2012).
\bibitem[CLO]{CLO}D. Cox, J. Little, D. O'Shea; \textit{Ideals, varieties and algorithms},
Undergraduate Texts in Mathematics, Second Edition, 1996.
\bibitem[E]{E}D. Eisenbud; \textit{Commutative Algebra with a View toward Algebraic Geometry},
Graduate Texts in Mathematics, Vol. 150, 1994.
\bibitem[G]{G}A. Grothendieck; \textit{\'{E}l\'ements de g\'eometrie alg\'ebrique IV, Quatri\`eme partie},
\'{E}tude locale des sch\'emas et des morphismes des sch\'emas, Publ. Math. IHES 32 (1967).
\bibitem[GS-1]{GS-1}G. Gonzalez-Sprinberg; \textit{Eventails en dimension 2 et transform\'{e} de Nash}, Publ.
de l'E.N.S., Paris (1977), 1-68.
\bibitem[H]{H}R. Hartshorne; \textit{Algebraic Geometry}, Graduate Texts in Mathematics, Vol. 52, 1977.
\bibitem[L]{L}Q. Liu; \textit{Algebraic Geometry and Arithmetic Curves}, Oxford Graduate Texts
in Mathematics, Vol. 6, 2002.
\bibitem[LT]{LT}D. Laksov, A. Thorup; \textit{Weierstrass points on schemes}, J. Reine Angew. Math.
no. 460, (1995).
\bibitem[No]{No}A. Nobile; \textit{Some properties of the Nash blowing-up}, Pacific Journal of Mathematics,
\textbf{60}, (1975), 297-305.
\bibitem[OZ]{OZ}A. Oneto, E. Zatini; \textit{Remarks on Nash blowing-up}, Rend. Sem. Mat. Univ. Politec.
Torino \textbf{49} (1991), no. 1, 71-82, Commutative algebra and algebraic geometry, II (Italian) (Turin 1990).
\bibitem[Sh]{Sh}D. O'Shea; \textit{Computing Limits of Tangent Spaces: Singularities,
Computation and Pedagogy}, Singularity Theory (Trieste, 1991), World Sci. Publ., pg. 549-573, 1995.
\bibitem[V]{V}O. Villamayor; \textit{On flattening of coherent sheaves and of projective morphisms},
Journal of Algebra, \textbf{295} (2006), no. 1, 119-140.
\bibitem[Y]{Y}T. Yasuda; \textit{Higher Nash blowups}, Compositio Math. \textbf{143} (2007), no. 6, 1493-1510.
\end{thebibliography}
\end{document}